\documentclass{siamart1116}

\usepackage{amsmath,amssymb,amstext}
\usepackage{bigints}

\makeatletter
\newcommand{\vast}{\bBigg@{4}}
\newcommand{\Vast}{\bBigg@{5}}
\makeatother

\usepackage[inline]{enumitem}

\usepackage{amsfonts}
\usepackage{graphicx}
\usepackage{epstopdf}
\ifpdf
  \DeclareGraphicsExtensions{.eps,.pdf,.png,.jpg}
\else
  \DeclareGraphicsExtensions{.eps}
\fi

\usepackage[caption=false]{subfig}

\numberwithin{theorem}{section}

\newcommand{\TheTitle}{Asymptotic Near-Minimaxity of\texorpdfstring{\\}{}
the Randomized Shiryaev--Roberts--Pollak Change-Point Detection Procedure\texorpdfstring{\\}{}
in Continuous Time}
\newcommand{\TheAuthors}{A. S. Polunchenko}

\headers{On Optimality of the Shiryaev--Roberts--Pollak Procedure}{\TheAuthors}

\title{{\TheTitle}\thanks{Submitted to the editors DATE.
\funding{This work was partially supported by the Simons Foundation via a Collaboration Grant in Mathematics under Award \#\,304574.}}}

\author{Aleksey S. Polunchenko%
    \thanks{Department of Mathematical Sciences, State University of New York at Binghamton, Binghamton, NY 13902--6000, USA
    (\email{aleksey@binghamton.edu}, \url{http://people.math.binghamton.edu/aleksey}).}
}

\usepackage{amsopn}

\renewcommand{\Pr}{\mathbb{P}} 
\DeclareMathOperator{\EV}{\mathbb{E}} 
\DeclareMathOperator{\Var}{\mathrm{Var}}

\newcommand{\W}{\mathop{}\!W}
\DeclareMathOperator{\Ei}{Ei}
\DeclareMathOperator{\E1}{E_1}

\newcommand{\T}{\tau}

\renewcommand{\le}{\leqslant} 
\renewcommand{\ge}{\geqslant}

\newcommand{\abs}[1]{\left\vert#1\right\vert}

\DeclareMathOperator{\One}{\mathchoice{\rm 1\mskip-4.2mu l}{\rm 1\mskip-4.2mu l}{\rm 1\mskip-4.6mu l}{\rm 1\mskip-5.2mu l}}
\newcommand{\indicator}[1]{{\One_{\left\{#1\right\}}}}

\ifpdf
\usepackage{hyperxmp}

\hypersetup{%
    bookmarks=true,         
    colorlinks=true,
    citecolor=blue,
    urlcolor=blue,
    unicode=false,          
    pdftoolbar=true,        
    pdfmenubar=true,        
    pdffitwindow=false,     
    pdfstartview={FitH},    
    pdftitle={Asymptotic near-minimaxity of the randomized Shiryaev--Roberts--Pollak change-point detection procedure in continuous time},    
    pdfauthor={Aleksey S. Polunchenko},     
    pdfsubject={Asymptotic near-minimaxity of the randomized Shiryaev--Roberts--Pollak change-point detection procedure in continuous time},   
    pdfcreator={Aleksey S. Polunchenko},   
    pdfproducer={MikTeX}, 
    pdfkeywords={Minimax optimality, Quasi-stationary distribution, Shiryaev--Roberts procedure, Sequential change-point detection}, 
    pdfnewwindow=true,      
    linkcolor=red,          
    filecolor=magenta      
}
\fi

\usepackage[ntheorem]{empheq}

\usepackage[norefs,nocites]{refcheck}
\usepackage{silence}
\begin{document}

\maketitle

\begin{abstract}
For the classical continuous-time quickest change-point detection problem it is shown that the randomized Shiryaev--Roberts--Pollak procedure is asymptotically nearly minimax-optimal (in the sense of Pollak~\cite{Pollak:AS85}) in the class of randomized procedures with vanishingly small false alarm risk. The proof is explicit in that all of the relevant performance characteristics are found analytically and in a closed form. The rate of convergence to the (unknown) optimum is elucidated as well. The obtained optimality result is a one-order improvement of that previously obtained by Burnaev et al.~\cite{Burnaev+etal:TPA2009} for the very same problem.
\end{abstract}



\section{Introduction, problem formulation and significance}
\label{sec:intro}
This work's focus is on the classical minimax change-point detection problem where the aim is to detect (in an optimal manner) a possible onset of a drift in ``live''-observed standard Brownian motion. More formally, suppose one is able to observe a ``live'' process, $(X_t)_{t\ge0}$, that is governed by the stochastic differential equation (SDE):
\begin{equation}\label{eq:BM-change-point-model}
dX_{t}
=
\mu\indicator{t>\theta}dt+dB_{t},\;\;t\ge0,\;\;\text{with}\;\;X_{0}=0,
\end{equation}
where $(B_t)_{t\ge0}$ is standard Brownian motion (i.e., $\EV[dB_t]=0$, $\EV[(dB_t)^2]=dt$, and $B_0=0$), $\mu\neq0$ is the {\em known} post-change drift magnitude, and $\theta\in[0,\infty]$ is the {\em unknown} (nonrandom) change-point; the notation $\theta=0$ ($\theta=\infty$) is to be understood as the case when $\EV[X_t]=0$ ($\EV[X_t]=\mu t$) for all $t>0$. One's objective is to establish {\em online} that the process' drift is no longer zero, and do so in an {\em optimal} fashion, i.e., as quickly as is possible within an {\it a~priori} set level of the false alarm risk.

Let $\Omega\triangleq\mathcal{C}[0,+\infty)$ be the space of continuous functions on $\mathbb{R}^+\triangleq[0,+\infty)$. Let $(\mathcal{F}_t)_{t\ge0}$, $\mathcal{F}_s\subseteq\mathcal{F}_t$ for $0\le s<t$, denote the filtration generated by $(X_t)_{t\ge0}$, i.e., $\mathcal{F}_t\triangleq\sigma(X_s,\;0\le s\le t)$ for $t>0$ and $\mathcal{F}_0$ is the trivial $\sigma$-algebra; note that $(\mathcal{F}_t)_{t\ge0}$ can be seen from~\eqref{eq:BM-change-point-model} to coincide with the filtration generated by the Brownian motion $(B_t)_{t\ge0}$ for any $\theta\in[0,\infty]$. Let $\mathcal{F}\triangleq\mathcal{F}_{\infty}\triangleq\vee_{t\ge0}\,\mathcal{F}_{t}$. With model~\eqref{eq:BM-change-point-model} placed on the filtered probability space $(\Omega,\mathcal{F},(\mathcal{F}_t)_{t\ge0},\Pr)$ any change-point detection procedure is a $(\mathcal{F}_t)_{t\ge0}$-measurable stopping time, $\T\triangleq\T(\omega)$, $\omega\in\Omega$, i.e., $\{\omega\colon\T(\omega)\le t\}\in\mathcal{F}_t$ for all $t\ge0$. The interpretation of $\T$ is that it is a rule to stop and declare that $(X_t)_{t\ge0}$ has (apparently) gained a drift of magnitude $\mu\neq0$. The decision made at stopping need not be correct. A ``good'' (i.e., optimal or nearly optimal) detection procedure $\T_{\mathrm{opt}}$ is one that minimizes (or nearly minimizes) the desired detection delay penalty, subject to a constraint on the false alarm risk. See, e.g.,~\cite{Tartakovsky+Moustakides:SA10,Polunchenko+Tartakovsky:MCAP2012},~\cite{Shiryaev:Book2011},~\cite[Chapter~VI]{Shiryaev:Book2017}, or~\cite[Part~II]{Tartakovsky+etal:Book2014} for a survey of the major existing optimality criteria. The specific one considered in this work is the minimax criterion of Pollak~\cite{Pollak:AS85}. See also~\cite{Burnaev+etal:TPA2009,Feinberg+Shiryaev:SD2006} where Pollak's~\cite{Pollak:AS85} criterion is referred to as ``Variant $(C)$'' of the quickest change-point detection problem. We now introduce it formally, following the original notation of~\cite{Burnaev+etal:TPA2009,Feinberg+Shiryaev:SD2006}.

Let $\Pr_{\theta}\triangleq\mathrm{Law}(X|\Pr,\theta)$ denote the probability measure (distribution law) induced by the observed process, $(X_t)_{t\ge0}$, under the assumption that the change-point, $\theta\in[0,+\infty]$, is {\em fixed}; note that $\Pr_{\infty}$ is the Wiener measure. Let $\EV_{\theta}$ represent the respective $\Pr_{\theta}$-expectation operator. Pollak's~\cite{Pollak:AS85} minimax version (or ``Variant ($C$)'' in the terminology used in~\cite{Burnaev+etal:TPA2009,Feinberg+Shiryaev:SD2006}) of the quickest change-point detection problem assumes that the false alarm risk is measured in terms of the classical Average Run Length (ARL) to false alarm metric defined as $\EV_{\infty}(\T)$, and the cost of a delay to (correct) detection is quantified via the largest (conditional) Average Detection Delay defined as
\begin{equation}\label{eq:SADD-ADDnu-def}
C(\T)
\triangleq
\sup_{\theta\ge0}C(\T,\theta)
\;\;
\text{where}
\;\;
C(\T,\theta)
\triangleq
\EV_{\theta}(\T-\theta|\T\ge\theta)
\;\;
\text{for}
\;\;
\theta\ge0,
\end{equation}
and the idea is to consider
\begin{equation}\label{eq:M-class-def}
\mathfrak{M}_{T}
\triangleq
\big\{\T\colon\EV_{\infty}(\T)=T\big\}
\;\;
\text{where}
\;\;
T>0
\;\;
\text{is given},
\end{equation}
i.e., the class of detection procedures (stopping times) $\T$ with the ARL to false alarm set at a given level $T>0$, and
\begin{equation}\label{eq:Pollak-minmax-problem}
\text{seek}
\;\;
\T_{\mathrm{opt}}\in\mathfrak{M}_{T}
\;\;
\text{such that}
\;\;
C(\T_{\mathrm{opt}})
=
C(T)
\;\;
\text{where}
\;\;
C(T)
\triangleq\inf_{\T\in\mathfrak{M}_{T}}C(\T),
\end{equation}
for {\em any} $T>0$.

Problem~\eqref{eq:Pollak-minmax-problem} is a major open problem in all of quickest change-point detection: although it has been attacked repeatedly (see, e.g.,~\cite{Yakir:AS97,Mei:AS2006,Tartakovsky+Polunchenko:IWAP10,Polunchenko+Tartakovsky:AS10}), its general solution is yet to be found, not only in the discrete-time setting, but in the continuous-time setting as well. The current ``favorite'' in the search for the solution seems to be the Generalized Shiryaev--Roberts (GSR) procedure of Moustakides et al.~\cite{Moustakides+etal:SS11}. The GSR procedure is a headstarted version of the classical quasi-Bayesian Shiryaev--Roberts (SR) procedure of Shiryaev~\cite{Shiryaev:SMD61,Shiryaev:TPA63} and Roberts~\cite{Roberts:T66}. Specifically, tailored to the Brownian motion scenario~\eqref{eq:BM-change-point-model}, the GSR procedure calls for stopping at:
\begin{equation}\label{eq:T-GSR-def}
\T_{A}^{(x)}
\triangleq
\inf\big\{t\ge0\colon \psi_{t}^{(x)}\ge A\big\}\;\text{such that}\;\inf\big\{\varnothing\big\}=\infty,
\end{equation}
where $A>0$ is a detection threshold (set in advance so as to keep the ``false positive'' risk tolerably low, i.e., to guarantee $\EV_{\infty}(\T_{A}^{(x)})=T$ for a given $T>0$), and the GSR statistic $(\psi_{t}^{(x)})_{t\ge0}$ is the diffusion process that solves the SDE:
\begin{equation}\label{eq:Rt_r-def}
d\psi_{t}^{(x)}
=
dt+\mu \psi_{t}^{(x)} dX_{t}
\;\;\text{with}\;\;
\psi_{0}^{(x)}\triangleq x\ge0,
\end{equation}
where $dX_{t}$ is as in~\eqref{eq:BM-change-point-model} above. The initial value $\psi_{0}^{(x)}\triangleq x$ is sometimes referred to as the headstart. The term ``{\em Generalized} Shiryaev--Roberts procedure'' appears to have been coined in~\cite{Tartakovsky+etal:TPA2012}, and was motivated by the fact that, in the no-headstart case, i.e., when $x=0$, the GSR procedure~\eqref{eq:T-GSR-def}-\eqref{eq:Rt_r-def} reduces to the classical SR procedure~\cite{Shiryaev:SMD61,Shiryaev:TPA63,Roberts:T66}. Continuing to adhere to the notation used in~\cite{Burnaev+etal:TPA2009,Feinberg+Shiryaev:SD2006}, we, too, shall denote the classical SR procedure's stopping time as $\T_{A}$ and its underlying statistic as $\psi_t$, i.e., define $\T_A\triangleq\T_{A}^{(0)}$ and $\psi_t\triangleq\psi_t^{(0)}$.

The reasons to suspect that the GSR procedure might actually solve problem~\eqref{eq:Pollak-minmax-problem} are three. The first reason is the result obtained (for the discrete-time analogue of the problem) in~\cite{Tartakovsky+Polunchenko:IWAP10,Polunchenko+Tartakovsky:AS10} where the GSR procedure with a ``finetuned'' headstart was {\em explicitly} demonstrated to be {\em exactly} Pollak-minimax in two specific (discrete-time) scenarios. The second reason is the {\em general} so-called almost Pollak-minimaxity of the GSR procedure (again, with a carefully designed headstart) established (in the discrete-time setting) in~\cite{Tartakovsky+etal:TPA2012}. More specifically, it was shown in~\cite{Tartakovsky+etal:TPA2012} that, if, for a given $T>0$, the GSR procedure's detection threshold $A=A_{T}>0$ and headstart $x=x_{T}\ge0$ are set so that $\T_{A_T}^{(x_T)}\in\mathfrak{M}_{T}$, but $x_T=o(A_T)$ in the sense that $\lim_{T\to+\infty}(x_{T}/A_{T})=0$, then
\begin{equation*}
C(\T_{A_{T}}^{(x_T)})-C(T)
=
o_{T}(1)
\;
\text{as}
\;
T\to+\infty,
\end{equation*}
where $C(\T)$ and $C(T)$ are as in~\eqref{eq:SADD-ADDnu-def} and~\eqref{eq:Pollak-minmax-problem}, respectively, and $o_T(1)\to0$ as $T\to+\infty$; see~\cite{Burnaev:ARSAIM2009} for an attempt to generalize this result to the continuous-time model~\eqref{eq:BM-change-point-model}. Since obviously $C(\T_{A_T}^{(x_T)})\to+\infty$ and $C(T)\to+\infty$ as $T\to+\infty$, the above is effectively saying that the GSR procedure is nearly Pollak-minimax-optimal, whenever the ARL to false alarm level $T>0$ is large. This is a strong optimality property known in the literature (see~\cite{Tartakovsky+etal:TPA2012}) as {\em order-three} asymptotic (as $T\to+\infty$) Pollak-minimaxity (or near Pollak-minimaxity).

However, the most important reason to study the GSR procedure deeper is the following: while the {\em general} solution $\T_{\mathrm{opt}}$ to Pollak's~\cite{Pollak:AS85} problem~\eqref{eq:Pollak-minmax-problem} is still unknown, there is a universal ``recipe'' (also proposed by Pollak~\cite{Pollak:AS85}) to achieve near Pollak-minimaxity, and the GSR procedure is the main ingredient of the ``recipe''. Specifically, Pollak's~\cite{Pollak:AS85} ingenious idea was to start the GSR statistic $(\psi_{t}^{(x)})_{t\ge0}$ off a random number sampled from the statistic's so-called quasi-stationary distribution (formally defined below). For the discrete-time version of the problem, Pollak~\cite{Pollak:AS85} was able to prove that such a randomized ``tweak'' of the GSR procedure is nearly Pollak-minimax; see also~\cite[Theorem~3.4]{Tartakovsky+etal:TPA2012}. It is to extend this result to the Brownian motion scenario~\eqref{eq:BM-change-point-model} that is the objective of this work.

The randomization of the GSR procedure's headstart necessitates the introduction of a probability space larger than the original $(\Omega,\mathcal{F},(\mathcal{F}_t)_{t\ge0},\Pr)$ constructed above. To that end, a suitable extension, which we shall denote $(\overline{\Omega},\overline{\mathcal{F}},(\overline{\mathcal{F}}_t)_{t\ge0},\overline{\Pr})$, has already been offered in~\cite{Burnaev+etal:TPA2009} and in~\cite[Chapter~II, Section~7]{Shiryaev:TPA63}, and the ingredients are:
\begin{enumerate*}
    \item $\overline{\Omega}\triangleq\Omega\times\tilde{\Omega}$, where $\tilde{\Omega}\triangleq[0,1]$;
    \item $\overline{\mathcal{F}}\triangleq\mathcal{F}\otimes\tilde{\mathcal{F}}$ and $\overline{\mathcal{F}}_t\triangleq\mathcal{F}_t\otimes\tilde{\mathcal{F}}$ for all $t\ge0$, where $\tilde{\mathcal{F}}\triangleq\mathcal{B}(\tilde{\Omega})$ is a Borel system of subsets on $\tilde{\Omega}$; and
    \item $\overline{\Pr}\triangleq\Pr\otimes \tilde{\Pr}$, where $\tilde{\Pr}$ is a Lebesgue measure on $(\tilde{\Omega},\tilde{\mathcal{F}})$.
\end{enumerate*} See also~\cite{Shiryaev:Book2017}.

To place Pollak's~\cite{Pollak:AS85} problem~\eqref{eq:Pollak-minmax-problem} on the new probabilistic basis $(\overline{\Omega},\overline{\mathcal{F}},(\overline{\mathcal{F}}_t)_{t\ge0},\overline{\Pr})$, define $\overline{\Pr}_{\theta}\triangleq\Pr_{\theta}\otimes\tilde{\Pr}$ for $\theta\in[0,+\infty]$, and let $\overline{\EV}_{\theta}$ denote the corresponding $\overline{\Pr}_{\theta}$-expectation operator. It is natural to measure the ARL to false alarm of a randomized procedure $\bar{\T}\triangleq\bar{\T}(\bar{\omega})$, $\bar{\omega}\triangleq(\omega,\tilde{\omega})\in\overline{\Omega}$, in terms of $\overline{\EV}_{\infty}(\bar{\T})$, and the worst Average Detection Delay via
\begin{equation}\label{eq:SADD-ADDnu-rnd-def}
\overline{C}(\bar{\T})
\triangleq
\sup_{\theta\ge0}\overline{C}(\bar{\T},\theta)
\;\;
\text{where}
\;\;
\overline{C}(\bar{\T},\theta)
\triangleq
\overline{\EV}_{\theta}(\bar{\T}-\theta|\bar{\T}\ge\theta)
\;\;
\text{for}
\;\;
\theta\ge0.
\end{equation}

Problem~\eqref{eq:Pollak-minmax-problem} can now be extended as follows. Consider
\begin{equation}\label{eq:M-rnd-class-def}
\overline{\mathfrak{M}}_{T}
\triangleq
\big\{\bar{\T}\colon\overline{\EV}_{\infty}(\bar{\T})=T\big\}
\;\;
\text{where}
\;\;
T>0
\;\;
\text{is given},
\end{equation}
i.e., the class of {\em randomized} detection procedures ({\em randomized} stopping times) $\bar{\T}$ with the ARL to false alarm set at a given level $T>0$, and
\begin{equation}\label{eq:Pollak-minmax-rnd-problem}
\text{seek}
\;\;
\bar{\T}_{\mathrm{opt}}\in\overline{\mathfrak{M}}_{T}
\;\;
\text{such that}
\;\;
\overline{C}(\bar{\T}_{\mathrm{opt}})
=
\overline{C}(T)
\;\;
\text{where}
\;\;
\overline{C}(T)
\triangleq\inf_{\bar{\T}\in\overline{\mathfrak{M}}_{T}}\overline{C}(\bar{\T}),
\end{equation}
for {\em any} $T>0$.

Problem~\eqref{eq:Pollak-minmax-rnd-problem}, just as problem~\eqref{eq:Pollak-minmax-problem}, is also still open, whether in discrete- or in continuous-time settings. It is referred to as ``Variant $(\overline{C})$'' of the quickest change-point detection problem in~\cite{Burnaev+etal:TPA2009}. While ``Variant $(C)$'' and ``Variant $(\overline{C})$'' are similar, they are {\em not} the same, because $\mathfrak{M}_{T}\subset\overline{\mathfrak{M}}_{T}$ for any fixed $T>0$, as can be seen from definitions~\eqref{eq:M-class-def} and~\eqref{eq:M-rnd-class-def}. Put another way, randomized detection procedures with the ARL to false alarm set at a prescribed level $T>0$ form a {\em larger} family than do their nonranomized counterparts with the same ARL to false alarm level $T>0$. As a result, even though neither $C(T)$ nor $\overline{C}(T)$ is known, it is apparent that $\overline{C}(T)\le C(T)$ for any fixed $T>0$. This work's specific focus is on problem~\eqref{eq:Pollak-minmax-rnd-problem}, and our ``course of attack'' is exactly the same as that of Pollak~\cite{Pollak:AS85} who considered the problem's discrete-time analogue and nearly solved it.

The main ingredient of Pollak's~\cite{Pollak:AS85} solution strategy is the quasi-stationary distribution of the SR statistic $(\psi_t)_{t\ge0}$. Formally, this distribution is defined as
\begin{equation}\label{eq:QSD-def}
Q_{A}(x)
\triangleq
\lim_{t\to+\infty}\Pr_{\infty}(\psi_{t}\le x|\T_{A}>t)
\;\;
\text{with}
\;\;
q_{A}(x)
\triangleq
\dfrac{d}{dx}Q_{A}(x)
\;\;
\text{where}
\;\;
x\in[0,A],
\end{equation}
and its existence follows, e.g., from the fundamental work of Mandl~\cite{Mandl:CMJ1961}; see also, e.g.,~\cite{Cattiaux+etal:AP2009} and~\cite[Section~7.8.2]{Collet+etal:Book2013}. The density $q_A(x)$ was studied in~\cite{Burnaev+etal:TPA2009} where the authors obtained a large-$A$ order-one expansion of $q_A(x)$. However, a more detailed investigation of the distribution and its properties was recently carried out in~\cite{Polunchenko:SA2017} where not only $Q_A(x)$ and $q_A(x)$ were both expressed analytically and in a closed form, but also the density $q_A(x)$ was shown to be unimodal, its entire moment series was computed, and more accurate (up to the third order) large-$A$ approximations of $q_A(x)$ were obtained as well. These results will play a critical role in the sequel.

The decision statistic behind Pollak's~\cite{Pollak:AS85} randomized version of the GSR procedure is the solution $(\psi_{t}^{*})_{t\ge0}\triangleq(\psi_{t}^{*}(\bar{\omega}))_{t\ge0}$ of the SDE:
\begin{equation}\label{eq:Rt_Q-def}
d\psi_{t}^{*}
=
dt+\mu\psi_{t}^{*} dX_{t}
\;
\text{with}
\;
\psi_{0}^{*}\propto Q_{A}(x),
\end{equation}
where $dX_{t}$ is as in~\eqref{eq:BM-change-point-model} and $Q_{A}(x)$ is defined in~\eqref{eq:QSD-def}. The corresponding stopping time is as follows:
\begin{equation}\label{eq:T-SRP-def}
\bar{\T}_{A}^{*}
\triangleq
\inf\big\{t\ge0\colon \psi_{t}^{*}\ge A\big\}
\;
\text{such that}
\;
\inf\big\{\varnothing\big\}=\infty,
\end{equation}
and we shall follow~\cite{Tartakovsky+etal:TPA2012} and refer to it as the (randomized) Shiryaev--Roberts--Pollak (SRP) procedure.

Pollak's~\cite{Pollak:AS85} motivation to introduce and study the SRP procedure~\eqref{eq:T-SRP-def}-\eqref{eq:Rt_Q-def} was to get the detection delay penalty $\overline{C}(\bar{\T},\theta)$ given by~\eqref{eq:SADD-ADDnu-rnd-def} independent of the change-point $\theta$, i.e., to achieve
\begin{equation*}
\overline{\EV}_{\theta}(\bar{\T}_{A}^{*})
\triangleq
\overline{C}(\bar{\T}_{A}^{*},0)
\equiv
\overline{C}(\bar{\T}_{A}^{*},\theta)
\triangleq
\overline{\EV}_{\theta}(\bar{\T}_{A}^{*}-\theta|\bar{\T}_{A}^{*}>\theta)
\;\;
\text{for any $\theta>0$ and $A>0$},
\end{equation*}
so that
\begin{equation}\label{eq:SRP-SADD-ADD-equalizer}
\sup_{\theta\ge0}\overline{\EV}_{\theta}(\bar{\T}_{A}^{*}-\theta|\bar{\T}_{A}^{*}>\theta)\triangleq\overline{C}(\bar{\T}_{A}^{*})
\equiv
\overline{C}(\bar{\T}_{A}^{*},0)
\triangleq
\overline{\EV}_{\theta}(\bar{\T}_{A}^{*})
\;\;
\text{for any $A>0$}.
\end{equation}

The foregoing delay-risk-equalization is a direct consequence of the fact that, by design, the process $(\psi_{t}^{*})_{t\ge0}$ has a time-invariant probabilistic structure, i.e., $\overline{\Pr}_{\infty}(\psi_{t}^{*}\le x|\bar{\T}_{A}^{*}>t)=\overline{\Pr}_{\infty}(\psi_{t}^{*}\le x|\psi_{s}^{*}<A,s\le t)=Q_A(x)$ for all $t\ge0$. A risk-equalizing property akin to~\eqref{eq:SRP-SADD-ADD-equalizer} is known in the general decision theory (see, e.g.,~\cite[Theorem~2.11.3]{Ferguson:Book1967}) to be a necessary condition for strict minimaxity. Hence the introduction of the SRP procedure by Pollak in~\cite{Pollak:AS85} was, in a way, Pollak's attempt to solve his very own minimax version of the quickest change-point detection problem, although considered only in the discrete-time setting. As was mentioned earlier, Pollak~\cite{Pollak:AS85} succeeded in proving only that the SRP procedure is asymptotically order-three Pollak-minimax; the result was recently reobtained in~\cite{Tartakovsky+etal:TPA2012} through a different approach. It is reasonable to expect the same result to hold for the continuous-time model~\eqref{eq:BM-change-point-model} as well. To that end, in~\cite{Burnaev+etal:TPA2009}, the SRP procedure $\bar{\T}_{A}^{*}$ given by~\eqref{eq:T-SRP-def}-\eqref{eq:Rt_Q-def} was shown to be asymptotically Pollak-minimax in the class of randomized procedures $\overline{\mathfrak{M}}_{T}$, but only up to the second order, i.e., the delay risk $\overline{C}(\bar{\T})$ is minimized up to an additive term that goes to a {\em positive} constant as the false alarm risk vanishes. That is, if, for a given $T>0$, the SRP procedure's threshold $A=A_{T}>0$ is set so that $\bar{\T}_{A}^{*}\in\overline{\mathfrak{M}}_{T}$, then
\begin{equation}\label{eq:SRP-opt-order2}
\overline{C}(\bar{\T}_{A_{T}}^{*})-\overline{C}(T)
=
{O}_{T}(1)
\;\;
\text{as}
\;\;
T\to+\infty,
\end{equation}
where ${O}_{T}(1)\to\text{const}>0$ as $T\to+\infty$.

We are now in a position to formally state the specific contribution of this work: it is shown in the sequel that the SRP procedure $\bar{\T}_{A}^{*}$ is almost Pollak-minimax among all reasonable randomized detection procedures. That is, if, for a given $T>0$, the SRP procedure's threshold $A=A_{T}>0$ is set so that $\bar{\T}_{A}^{*}\in\overline{\mathfrak{M}}_{T}$, then
\begin{equation}\label{eq:SRP-opt-order3}
\overline{C}(\bar{\T}_{A_{T}}^{*})-\overline{C}(T)
=
o_{T}(1)
\;\;
\text{as}
\;\;
T\to+\infty,
\end{equation}
where we reiterate that $o_{T}(1)\to0$ as $T\to+\infty$. This is a one-order improvement of~\eqref{eq:SRP-opt-order2} previously proved in~\cite{Burnaev+etal:TPA2009}. Moreover, it is also shown in the sequel that the ``$o_{T}(1)$'' sitting in the right-hand side of~\eqref{eq:SRP-opt-order3} vanishes no slower than $1/\sqrt{\mu^2 T}$ as $T\to+\infty$.

\section{Summary of relevant prior results}\label{sec:preliminaries}
Our proof of~\eqref{eq:SRP-opt-order3} utilizes certain results established in the literature earlier. Hence, to streamline the proof, this section summarizes the relevant prior results. To that end, the latter can be divided up into two categories. Category 1 includes results that concern properties of the GSR procedure~\eqref{eq:T-GSR-def}-\eqref{eq:Rt_r-def}, including the classical SR procedure~\cite{Shiryaev:SMD61,Shiryaev:TPA63,Roberts:T66} as its particular case. These results are all due to A.N.~Shiryaev and his co-authors. By contrast, Category 2 is comprised of results on properties of the {\em randomized}  SRP procedure~\eqref{eq:T-SRP-def}-\eqref{eq:Rt_Q-def}. These results all come from~\cite{Polunchenko:SA2017} and concern the SR statistic's quasi-stationary distribution defined in~\eqref{eq:QSD-def}.

We start by going over the first group of results. The first result is the fact that, for any given $T>0$, the {\em unknown} optimal delay risks $C(T)$ and $\overline{C}(T)$ defined in~\eqref{eq:Pollak-minmax-problem} and in~\eqref{eq:Pollak-minmax-rnd-problem}, respectively, both permit an {\em explicitly computable} lowerbound. Specifically, the following inequalities hold true
\begin{equation}\label{eq:LwrBnd-optSADD-ineq}
B(T)
\le
C(T)
\;\;
\text{and}
\;\;
\overline{B}(T)
\le
\overline{C}(T)
\;\;
\text{for any}
\;\;
T>0,
\end{equation}
where
\begin{equation*}
B(T)
\triangleq
\inf_{\T\in\mathfrak{M}_{T}}\dfrac{1}{T}\int_{0}^{\infty}\EV_{\theta}(\T-\theta)^{+}d\theta
\;\;
\text{and}
\;\;
\overline{B}(T)
\triangleq
\inf_{\bar{\T}\in\overline{\mathfrak{M}}_{T}}\dfrac{1}{T}\int_{0}^{\infty}\overline{\EV}_{\theta}(\bar{\T}-\theta)^{+}d\theta,
\end{equation*}
where $x^+\triangleq\max\{0,x\}$; cf.~\cite{Feinberg+Shiryaev:SD2006,Burnaev+etal:TPA2009}. The quantities $B(T)$ and $\overline{B}(T)$ are the optimal generalized Bayesian risks: they quantify the delay cost when $\theta$ is random and sampled from an improper uniform distribution on $[0,+\infty)$. See~\cite{Shiryaev:Bachelier2002,Shiryaev:MathEvents2006,Shiryaev+Zryumov:Khabanov2010}.

A remarkable fact about $B(T)$ and $\overline{B}(T)$ is that $B(T)=\overline{B}(T)$ for any $T>0$. See~\cite[Chapter~II, Section~7]{Shiryaev:Book78} and~\cite[Chapter~VI]{Shiryaev:Book2017}. Moreover, both $B(T)$ and $\overline{B}(T)$ permit the following explicit (and amenable to numerical evaluation) representation:
\begin{equation}\label{eq:LwrBnd-SR-formula}
B(T)
=
\overline{B}(T)
=
\dfrac{2}{\mu^2}\left\{F\left(\dfrac{2}{\mu^2 T}\right)-1+\dfrac{2}{\mu^2 T}\bigintsss_{0}^{T}F\left(\dfrac{2}{\mu^2 x}\right)\dfrac{dx}{x}\right\},
\end{equation}
where
\begin{equation}\label{eq:F-func-def}
F(x)
\triangleq
e^{x}\E1(x)
\end{equation}
with
\begin{equation}\label{eq:ExpInt-def}
\E1(x)
\triangleq
\int_{x}^{+\infty}e^{-t}\,\dfrac{dt}{t},\;\; x>0,
\end{equation}
being the so-called exponential integral, a special function often also denoted as $-\Ei(-x)$; see, e.g.,~\cite{Geller+Ng:JRNBS1969} and \cite[Chapter~5]{Abramowitz+Stegun:Handbook1964}. Formula~\eqref{eq:LwrBnd-SR-formula} is a straightforward generalization of~\cite[Theorem~2.3]{Feinberg+Shiryaev:SD2006} which gives the formula only in the special case of $\mu=\sqrt{2}$. It is now apparent (cf.~\cite[Theorem~4.4]{Feinberg+Shiryaev:SD2006}) that $B(T)=\overline{B}(T)\le\overline{C}(T)\le C(T)$ for any $T>0$.

The third result is a classical property of the GSR procedure $\T_{A}^{(x)}$ defined in~\eqref{eq:T-GSR-def}, namely that
\begin{equation}\label{eq:GSR-ARL-formula}
\EV_{\infty}(\T_{A}^{(x)})
=
A-x,
\;\;
\text{for any}
\;\;
x\in[0,A],
\;\;
\text{with}
\;\;
A>0;
\end{equation}
cf., e.g.,~\cite[p.~530]{Burnaev+etal:TPA2009}, although the result is likely to have been first discovered by A.N. Shiryaev in the early 1960's. In the special case of no headstart, formula~\eqref{eq:GSR-ARL-formula} reduces to the equally well-known fact that $\EV_{\infty}(\T_{A}^{(x)})=A$ for any $A>0$; incidentally, the formula $\EV_{\infty}(\T_{A}^{(x)})=A$ is involved in the derivation of~\eqref{eq:LwrBnd-SR-formula}. It will also prove useful to point out that one way to arrive at~\eqref{eq:GSR-ARL-formula} is to notice that the process $(\psi_{t}^{(x)}-t-x)_{t\ge0}$ is a zero-mean $\Pr_{\infty}$-martingale, i.e., $\EV_{\infty}(\psi_{t}^{(x)}-t-x)=0$ for any $t\ge0$ and $x\in[0,A]$, and then invoke Doob's optional stopping theorem to deduce that $\EV_{\infty}(\T_{A}^{(x)})=\EV_{\infty}(\psi_{\T_{A}^{(x)}}^{(x)})-x$, and then finally make the transition to~\eqref{eq:GSR-ARL-formula} by arguing that the GSR statistic $(\psi_{t}^{(x)})_{t\ge0}$ reaches any level $A>0$ almost surely, so that $\psi_{\T_{A}^{(x)}}^{(x)}=A$ with probability 1, under any measure $\Pr_{\theta}$.

The forth result is yet another classical property of the GSR procedure, namely that
\begin{equation}\label{eq:GSR-ADD0-formula}
\EV_{0}(\T_{A}^{(x)})
\triangleq
C(\T_{A}^{(x)},0)
=
\dfrac{2}{\mu^2}\left\{F\left(\dfrac{2}{\mu^2 A}\right)-F\left(\dfrac{2}{\mu^2 x}\right)\right\},
\;\;
x\in[0,A],
\end{equation}
where $F(x)$ is the function introduced in~\eqref{eq:F-func-def}. This formula is a trivial generalization of~\cite[Lemma~3.3]{Feinberg+Shiryaev:SD2006} where it was established in the special case of $\mu=\sqrt{2}$. It is worth mentioning that formula~\eqref{eq:GSR-ADD0-formula}, just as formula~\eqref{eq:GSR-ARL-formula}, is also involved in the proof of~\eqref{eq:LwrBnd-SR-formula}.

The fifth and final result to go into the first category is the assertion that
\begin{equation}\label{eq:LwrBnd-last-int-term-asymp}
\dfrac{2}{\mu^2 T}\bigintsss_{0}^{T}F\left(\dfrac{2}{\mu^2 x}\right)\dfrac{dx}{x}
=
{O}\left(\dfrac{\log^2(\mu^2T)}{\mu^2T}\right)
\;
\text{as}
\;
T\to+\infty,
\end{equation}
which follows from~\cite[Formulae (2.33) and (2.34), p. 456]{Feinberg+Shiryaev:SD2006}. It is also noteworthy that the quantity sitting in the left-hand side of~\eqref{eq:LwrBnd-last-int-term-asymp} is nonnegative for any $T>0$.

We now switch attention to the second group of results, which all revolve around the formula
\begin{equation}\label{eq:SRP-ADD-def}
\overline{C}(\bar{\T}_{A}^{*})
=
\bigintsss_{0}^{A}C(\T_{A}^{(x)},0)\,q_A(x)\,dx,
\end{equation}
where $C(\T_{A}^{(x)},0)\triangleq\EV_{0}(\T_A^{(x)})$ is given explicitly by~\eqref{eq:GSR-ADD0-formula} above, and $q_A(x)$ is the pdf of the GSR statistic's quasi-stationary distribution formally defined in~\eqref{eq:QSD-def}. It is evident from formula~\eqref{eq:SRP-ADD-def} getting $\overline{C}(\bar{\T}_{A}^{*})$ expressed explicitly is impossible without a closed-form expression for $q_A(x)$. Such an expression was recently obtained in~\cite{Polunchenko:SA2017}, and it is presented next.

Specifically, in~\cite{Polunchenko:SA2017}, it was shown that, for any fixed detection threshold $A>0$, the density $q_A(x)$ is given by
\begin{equation}\label{eq:QST-pdf-answer}
q_A(x)
=
\dfrac{\dfrac{1}{x}\,e^{-\tfrac{1}{\mu^2 x}}\W_{1,\tfrac{\xi}{2}}\left(\dfrac{2}{\mu^2 x}\right)}{e^{-\tfrac{1}{\mu^2 A}}\W_{0,\tfrac{\xi}{2}}\left(\dfrac{2}{\mu^2 A}\right)}\,\indicator{x\in[0,A]},
\end{equation}
while the respective cdf $Q_A(x)$ is given by
\begin{empheq}[%
    left={%
        Q_A(x)=%
    \empheqlbrace}]{align}\nonumber
&1,\;\text{for $x\ge A$;}\label{eq:QST-cdf-answer}\\
&\dfrac{e^{-\tfrac{1}{\mu^2 x}}\W_{0,\tfrac{\xi}{2}}\left(\dfrac{2}{\mu^2 x}\right)}{e^{-\tfrac{1}{\mu^2 A}}\W_{0,\tfrac{\xi}{2}}\left(\dfrac{2}{\mu^2 A}\right)},\;\text{for $x\in[0,A)$;}\\\nonumber
&0,\;\text{otherwise},
\end{empheq}
where
\begin{equation}\label{eq:xi-def}
\xi
\equiv
\xi(\lambda)
\triangleq
\sqrt{1-\dfrac{8}{\mu^2}\lambda},
\end{equation}
with $\lambda\equiv\lambda_{A}$ being the smallest nonnegative solution of the (always consistent) equation
\begin{equation}\label{eq:Whit1-eigval-eqn}
\W_{1,\tfrac{\xi(\lambda)}{2}}\left(\dfrac{2}{\mu^2 A}\right)
=
0,
\end{equation}
and $\W_{a,b}(z)$ is the standard notation for the special function known the Whittaker $\W$ function. The latter is defined as one of the two fundamental solutions $w(z)$ of the so-called Whittaker~\cite{Whittaker:BAMS1904} equation
\begin{equation}\label{eq:Whittaker-eqn}
\dfrac{\partial^2}{\partial z^2}\,w(z)+\left(-\dfrac{1}{4}+\dfrac{a}{z}+\dfrac{1/4-b^2}{z^2}\right)w(z)
=
0,
\;\;
z\in\mathbb{C},
\end{equation}
where $a,b\in\mathbb{C}$ are parameters. The second fundamental solution of the Whittaker equation~\eqref{eq:Whittaker-eqn} is known as the Whittaker $M$ function, and it is conventionally denoted as $M_{a,b}(z)$. A distinguishing feature of $M_{a,b}(z)$ is that, unlike $W_{a,b}(z)$, it does not exist when $2b=-1,-2,-3,\ldots$, and has to be regularized. See~\cite{Slater:Book1960} and~\cite{Buchholz:Book1969} for an extensive study of the Whittaker $\W$ and $M$ functions.

Formulae~\eqref{eq:QST-pdf-answer} and~\eqref{eq:QST-cdf-answer}, including condition~\eqref{eq:Whit1-eigval-eqn}, all put together make up the first result to go into the second group of results. In a nutshell, the formulae are the solution of a certain Sturm--Liouville problem, and $\lambda$ is the smallest eigenvalue of the corresponding Sturm--Liouville operator. See~\cite[Section~2]{Polunchenko:SA2017} and~\cite[Section~3]{Burnaev+etal:TPA2009}. We also remark parenthetically that the original notation used in~\cite{Polunchenko:SA2017} is $-\lambda\;(\le0)$ rather than $\lambda\;(\ge0)$. We made this flip in the sign here entirely for convenience. We also note that $\lambda$, as a solution of equation~\eqref{eq:Whittaker-eqn}, is dependent on $A>0$, and throughout what is to follow, where necessary, we shall emphasize this dependence via the notation $\lambda_A$.

The next result to go into Category 2 is a result also obtained in~\cite{Polunchenko:SA2017}, and it concerns the quasi-stationary distribution's moments. Specifically, as shown explicitly in~\cite{Polunchenko:SA2017}, if $Z$ is a random variable sampled from a population with the pdf $q_A(x)$ given by~\eqref{eq:QST-pdf-answer}, then $\EV[Z]=A-1/\lambda$, and
\begin{equation}\label{eq:QSD-Var-formula}
\Var[Z]
=
\dfrac{\lambda-\mu^2(A\lambda-1)^2}{\lambda^2(\mu^2+\lambda)},
\end{equation}
where we reiterate that $\lambda\ge0$ is the largest (nonnegative) solution of equation~\eqref{eq:Whit1-eigval-eqn}. The foregoing formulae for the first moment and variance of the quasi-stationary distribution were obtained directly from~\eqref{eq:QST-pdf-answer} using properties of the Whittaker $\W$ function.

The formula $\EV[Z]=A-1/\lambda$ can also be derived from~\eqref{eq:GSR-ARL-formula}. To that end, the key is to recall that the $\overline{\Pr}_{\infty}$-distribution of the SRP procedure's stopping time $\bar{\T}_{A}^{*}$ is exactly exponential with parameter $\lambda$, so that $\overline{\EV}_{\infty}(\bar{\T}_{A}^{*})=1/\lambda$. Consequently, if one now averages~\eqref{eq:GSR-ARL-formula} through with respect to $x$ assuming that $x\propto q_A(x)$, then $\EV[Z]=A-1/\lambda$ will follow easily. Since $Z$ is a nonnegative random variable (taking values in the interval $[0,A]$), it further follows that
\begin{equation}\label{eq:lambda-order-one-ineq}
(0<)\;
\dfrac{1}{A}
\le
\lambda
,
\;
A>0,
\end{equation}
which can be interpreted thus: to achieve the same ARL to false alarm level, the SRP procedure requires a higher detection threshold than does the classical SR procedure. This is an anticipated consequence of the randomization used to initialize the SRP statistic.

A more useful inequality can be gleaned from the formula~\eqref{eq:QSD-Var-formula} for $\Var[Z]$. Specifically, by requiring the fraction in the right-hand side of~\eqref{eq:QSD-Var-formula} to be no less than zero, after some elementary algebra, one obtains
\begin{equation*}
(0<)\;
\dfrac{1}{A}+\dfrac{1-\sqrt{4\mu^2 A+1}}{2\mu^2 A^2}
\le
\lambda_A
\le
\dfrac{1}{A}+\dfrac{1+\sqrt{4\mu^2 A+1}}{2\mu^2 A^2},
\end{equation*}
which, in view of~\eqref{eq:lambda-order-one-ineq}, can be ``tighten up'' from below to the double inequality
\begin{equation}\label{eq:lambda-dbl-ineq}
(0<)\;
\dfrac{1}{A}
\le
\lambda_A
\le
\dfrac{1}{A}+\dfrac{1+\sqrt{4\mu^2 A+1}}{2\mu^2 A^2};
\end{equation}
cf.~\cite{Polunchenko:SA2017}. Though somewhat conservative (especially when $A$ is small), this double inequality will prove good enough for our purposes. An important implication of the inequality is that
\begin{equation}\label{eq:lambda-large-A-asymp}
\lambda_A
=
\dfrac{1}{A}
+
{O}\left(\dfrac{1}{\abs{\mu}A^{3/2}}\right)
\;\;
\text{as}
\;\;
A\to+\infty,
\end{equation}
which is a generalization and also a refinement of the conclusion that
\begin{equation}\label{eq:lambda-large-A-asymp-Shiryaev}
\lambda_A
=
\dfrac{6e}{A}
+
{O}\left(\dfrac{1}{A^{2}}\right)
\;\;
\text{as}
\;\;
A\to+\infty,
\end{equation}
made earlier in~\cite[p.~528]{Burnaev+etal:TPA2009} under the assumption that $\mu=\sqrt{2}$. Recalling now that $\overline{\EV}_{\infty}(\bar{\T}_{A}^{*})=1/\lambda_A$ and that $\EV_{\infty}(\T_{A})=A$ it is direct to see from~\eqref{eq:lambda-large-A-asymp} that
\begin{equation*}
\dfrac{\overline{\EV}_{\infty}(\bar{\T}_{A}^{*})}{\EV_{\infty}(\T_{A})}\to 1
\;\;
\text{as}
\;\;
A\to+\infty,
\end{equation*}
i.e., the ARL to false alarm of the SRP procedure and that of the SR procedure with the same threshold $A>0$ are approximately the same, whenever $A$ is large. Such a strong conclusion clearly does not follow from~\eqref{eq:lambda-large-A-asymp-Shiryaev}.

\section{Proof of asymptotic near Pollak-minimaxity of the randomized SRP procedure}\label{sec:main-result}
Let us now fix the SRP procedure's ARL to false alarm level at a given $T>0$, i.e., suppose that $T>0$ is given and that the SRP procedure's threshold $A=A_{T}>0$ is such that $\overline{\EV}_{\infty}(\bar{\T}_{A_{T}}^{*})=T$, which, by definition~\eqref{eq:M-rnd-class-def}, is equivalent to $\bar{\T}_{A_{T}}^{*}\in\mathfrak{M}_{T}$.

The gist of our strategy to prove~\eqref{eq:SRP-opt-order3}, i.e., the desired near Pollak-minimaxity of the SRP procedure, is to show that
\begin{equation*}
\overline{C}(\bar{\T}_{A_T}^{*})-\overline{B}(T)
=
o_{T}(1)
\;\;
\text{as}
\;\;
T\to+\infty,
\end{equation*}
where $o_T(1)\to0$ as $T\to+\infty$. The reason this is a plausible approach is because of the ``sandwich'' inequality $\overline{B}(T)\le\overline{C}(T)\le\overline{C}(\bar{\T}_{A_{T}}^{*})$ implied by~\eqref{eq:LwrBnd-optSADD-ineq} together with the obvious $\overline{C}(T)\le\overline{C}(\bar{\T}_{A_{T}}^{*})$.

Since $\overline{B}(T)$ is given explicitly by~\eqref{eq:LwrBnd-SR-formula}, the strategy could work if $\overline{C}(\bar{\T}_{A_{T}}^{*})$ were also expressed in a closed-form. To that end, the problem is that even though all of the ingredients, viz.~\eqref{eq:GSR-ADD0-formula},~\eqref{eq:F-func-def}, and~\eqref{eq:QST-pdf-answer} with~\eqref{eq:xi-def} and~\eqref{eq:Whit1-eigval-eqn}, required to find $\overline{C}(\bar{\T}_{A_{T}}^{*})$ in a closed-form through~\eqref{eq:SRP-ADD-def} {\em are} available, the actual evaluation of the integral in the right-hand side of~\eqref{eq:SRP-ADD-def} is hampered by the presence of special functions in the integrand. To boot, getting $\overline{C}(\bar{\T}_{A_{T}}^{*})$ expressed explicitly in just any form will not do: it needs to be in a form similar to that given by~\eqref{eq:LwrBnd-SR-formula} for $\overline{B}(T)$, so that the difference $\overline{C}(\bar{\T}_{A_{T}}^{*})-\overline{B}(T)\;(>0)$ can be conveniently upperbounded. All these challenges are overcome in the following lemma.
\begin{lemma}\label{LEM:SRP-SADD-LWRBND}
For any given value $A>0$ of the SRP procedure's detection threshold, the procedure's delay risk $\overline{C}(\bar{\T}_{A}^{*})$ permits the representation:
\begin{equation}\label{eq:SRP-SADD-LwrBnd-form}
\overline{C}(\bar{\T}_{A}^{*})
=
\dfrac{2}{\mu^2}\left\{F\left(\dfrac{2}{\mu^2 A}\right)-1+\dfrac{2\lambda}{\mu^2}\bigintsss_{0}^{A}F\left(\dfrac{2}{\mu^2 x}\right)Q_{A}(x)\,\dfrac{dx}{x}\right\},
\end{equation}
where $F(x)$ is as in~\eqref{eq:F-func-def}, $Q_A(x)$ is given by~\eqref{eq:QST-cdf-answer}, and $\lambda\equiv\lambda_A$ is determined by the equation~\eqref{eq:Whit1-eigval-eqn}. Note that formula~\eqref{eq:SRP-SADD-LwrBnd-form} is {\em not} an inequality.
\end{lemma}
\begin{proof}
The whole problem---in view of formulae~\eqref{eq:GSR-ADD0-formula},~\eqref{eq:F-func-def}, \eqref{eq:SRP-ADD-def}, and \eqref{eq:QST-pdf-answer}---is effectively to find the integral
\begin{equation}\label{eq:I-int-def}
I
\triangleq
\bigintsss_{\tfrac{2}{\mu^2 A}}^{+\infty} e^{\tfrac{y}{2}}\E1(y)\W_{1,\tfrac{\xi}{2}}(y)\,\dfrac{dy}{y},
\end{equation}
where $\xi\equiv\xi(\lambda)$ is given by~\eqref{eq:xi-def} with $\lambda\ge0$ determined by~\eqref{eq:Whit1-eigval-eqn}; incidentally, condition~\eqref{eq:Whit1-eigval-eqn} will prove crucial in the evaluation of $I$. It is worth reminding that $\E1(z)$ denotes the exponential integral~\eqref{eq:ExpInt-def}, while $\W_{a,b}(z)$ denotes the Whittaker $\W$ function formally defined as a fundamental solution of the Whittaker equation~\eqref{eq:Whittaker-eqn}.

The integral $I$ introduced in~\eqref{eq:I-int-def} can be found using integration by parts. Specifically, observe that if
\begin{equation*}
u\triangleq
y\E1(y)
\;\;\text{and}\;\;
dv\triangleq
e^{\tfrac{y}{2}}\W_{0,\tfrac{\xi}{2}}(y)\,\dfrac{dy}{y^2},
\end{equation*}
then
\begin{equation*}
du
=
\left[\E1(y)-e^{-y}\right]dy
\;\;\text{and}\;\;
v
=
\dfrac{\mu^2}{2\lambda y}\,e^{\tfrac{y}{2}}\W_{1,\tfrac{\xi}{2}}(y)
\end{equation*}
where the formula for $du$ is a trivial consequence of~\eqref{eq:ExpInt-def} while the formula for $v$ is due to~\eqref{eq:xi-def} and the Whittaker $W$ function's general differential property
\begin{equation*}
\dfrac{\partial}{\partial z}\left[e^{\tfrac{z}{2}}\,z^{-k}\W_{k,b}(z)\right]
=
\left(b-k+\dfrac{1}{2}\right)\left(b+k-\dfrac{1}{2}\right)e^{\tfrac{z}{2}}\,z^{-k-1}\W_{k-1,b}(z),
\end{equation*}
given, e.g., by~\cite[Identity~2.4.21,~p.~25]{Slater:Book1960}. Therefore, in view of~\eqref{eq:Whit1-eigval-eqn}, plus the large-argument asymptotic of the Whittaker $W$ function
\begin{equation*}
\W_{a,b}(z)
=
z^{a}e^{-\tfrac{z}{2}}\left[1+{O}\left(\dfrac{1}{z}\right)\right]\;\text{as}\;|z|\to+\infty,\;\text{for any $a,b\in\mathbb{C}$},
\end{equation*}
established, e.g., in~\cite[Section~16.3]{Whittaker+Watson:Book1927}, and
\begin{equation*}
\lim_{x\to+\infty} x^{a}\E1(x)
=
0,
\;\;
\text{for any}
\;\;
a\in\mathbb{R},
\end{equation*}
given, e.g., by~\cite[Identity~3.2.5,~p.~193]{Geller+Ng:JRNBS1969}, it follows that
\begin{equation*}
-\dfrac{2\lambda}{\mu^2}\bigintsss_{\tfrac{2}{\mu^2 A}}^{+\infty}e^{\tfrac{y}{2}}\E1(y)\W_{0,\tfrac{\xi}{2}}(y)\,\dfrac{dy}{y}
=
I-\bigintsss_{\tfrac{2}{\mu^2 A}}^{+\infty} e^{-\tfrac{y}{2}}\W_{1,\tfrac{\xi}{2}}(y)\dfrac{dy}{y},
\end{equation*}
whence, using~\cite[Integral~7.623.7,~p.~824]{Gradshteyn+Ryzhik:Book2007}, i.e., the definite integral identity
\begin{equation*}
\begin{split}
\bigintssss_{1}^{+\infty}(x-1)^{c-1}&x^{a-c-1}\,e^{-\tfrac{zx}{2}}\W_{a,b}(zx)\,dx\\
&\quad=
\Gamma(c)\,e^{-\tfrac{z}{2}}\W_{a-c,b}(z),
\;\text{provided $\Re(c)>0$ and $\Re(z)>0$},
\end{split}
\end{equation*}
where $\Gamma(z)$ denotes the Gamma function (see, e.g.,~\cite[Chapter~6]{Abramowitz+Stegun:Handbook1964}), it further follows that
\begin{equation*}
I
=
e^{-\tfrac{1}{\mu^2 A}}\W_{0,\tfrac{\xi}{2}}\left(\dfrac{2}{\mu^2 A}\right)
-
\dfrac{2\lambda}{\mu^2}\bigintsss_{\tfrac{2}{\mu^2 A}}^{+\infty}e^{\tfrac{y}{2}}\E1(y)\W_{0,\tfrac{\xi}{2}}(y)\,\dfrac{dy}{y},
\end{equation*}
which, recalling~\eqref{eq:F-func-def} and~\eqref{eq:QST-cdf-answer}, can be seen to give the sought identity~\eqref{eq:SRP-SADD-LwrBnd-form}.
\end{proof}

We hasten to note the similarity between the right-hand side of~\eqref{eq:SRP-SADD-LwrBnd-form} and the right-hand side of~\eqref{eq:LwrBnd-SR-formula}. It is to achieve this similarity that is the whole point of Lemma~\ref{LEM:SRP-SADD-LWRBND}. Proving~\eqref{eq:SRP-opt-order3} is all downhill from now.

\begin{lemma}\label{lem:SRP-AT-dbl-ineq}
If, for a given ARL to false alarm level $T>0$, the SRP procedure's detection threshold $A\triangleq A_{T}>0$ is set so that $\bar{\T}_{A_{T}}^{*}\in\overline{\mathfrak{M}}_{T}$, then
\begin{equation}\label{eq:SRP-AT-dbl-ineq}
T
\le
A_{T}
\le
T+\dfrac{\sqrt{T}}{\abs{\mu}},
\end{equation}
where $\mu\neq0$ is the anticipated post-change drift magnitude in the Brownian motion model~\eqref{eq:BM-change-point-model}.
\end{lemma}
\begin{proof}
It suffices to recall that $\overline{\EV}_{\infty}(\bar{\T}_{A}^{*})=1/\lambda_{A}$, so that $\bar{\T}_{A}^{*}\not\in\overline{\mathfrak{M}}_{T}$ unless $A=A_{T}>0$ is such that $\lambda_{A_{T}}=1/T$, and then solve the double inequality~\eqref{eq:lambda-dbl-ineq} for $A_{T}$ under the assumption that $\lambda_{A_{T}}=1/T$.
\end{proof}

At this point note that, in view~\eqref{eq:SRP-AT-dbl-ineq}, if $\lambda_{A_{T}}=1/T$, then $A_{T}\ge T$, so that~\eqref{eq:SRP-SADD-LwrBnd-form} can be rewritten as
\begin{multline*}
\overline{C}(\bar{\T}_{A_{T}}^{*})
=
\dfrac{2}{\mu^2}\vast\{F\left(\dfrac{2}{\mu^2 A_{T}}\right)-1\\
+
\dfrac{2}{\mu^2T}\left[\bigintsss_{0}^{T}F\left(\dfrac{2}{\mu^2 x}\right)Q_{A_{T}}(x)\,\dfrac{dx}{x}+\bigintsss_{T}^{A_{T}}F\left(\dfrac{2}{\mu^2 x}\right)Q_{A_{T}}(x)\,\dfrac{dx}{x}\right]\vast\},
\end{multline*}
which is a form convenient enough to subtract off $\overline{B}(T)$ given by~\eqref{eq:LwrBnd-SR-formula}, and proceed to constructing a suitable upperbound for the difference $\overline{C}(\bar{\T}_{A_{T}}^{*})-\overline{B}(T)$. Specifically, recalling that $Q_A(x)$ is a cdf, so that $0\le Q_A(x)\le 1$ for any $x\in\mathbb{R}$ and any $A>0$, we arrive at the inequality
\begin{equation*}
(0\le)\;\overline{C}(\bar{\T}_{A_{T}}^{*})
-
\overline{B}(T)
\le
\dfrac{2}{\mu^2}\bigl\{J_1(T)+J_2(T)\bigr\},
\end{equation*}
where
\begin{equation*}
J_1(T)
\triangleq
F\left(\dfrac{2}{\mu^2 A_{T}}\right)-F\left(\dfrac{2}{\mu^2 T}\right)
\;\;
\text{and}
\;\;
J_2(T)
\triangleq
\dfrac{2}{\mu^2T}\bigintsss_{T}^{A_{T}}F\left(\dfrac{2}{\mu^2 x}\right)\dfrac{dx}{x}
\end{equation*}
so that if we could show that $J_1(T)\to0$ and $J_2(T)\to0$ as $T\to+\infty$, then the desired result~\eqref{eq:SRP-opt-order3} would follow at once.


To show that $J_1(T)\to0$ as $T\to+\infty$, observe from~\eqref{eq:F-func-def} and~\eqref{eq:ExpInt-def} that $F'(x)=F(x)-1/x$, and then because $(0<)\;e^x\E1(x)\le 1/x$ for $x>0$, as given by~\cite[Inequality~5.1.19,~p.~229]{Abramowitz+Stegun:Handbook1964}, conclude that $F(x)$ is a nonincreasing function of $x>0$. This implies that $J_1(T)>0$ for all $T>0$, and, more importantly, using the Mean Value Theorem we also have
\begin{multline*}
(0<)\;F\left(\dfrac{2}{\mu^2 A_{T}}\right)
-
F\left(\dfrac{2}{\mu^2 T}\right)
=
\left[F(z_{T})-\dfrac{1}{z_{T}}\right]\left(\dfrac{2}{\mu^2 A_{T}}-\dfrac{2}{\mu^2T}\right),\\
\text{for some}
\;\;
z_{T}\in\left(\dfrac{2}{\mu^2 A_{T}},\dfrac{2}{\mu^2 T}\right),
\end{multline*}
whence, in view of~\eqref{eq:SRP-AT-dbl-ineq}, the fact trivially seen from~\eqref{eq:F-func-def} and~\eqref{eq:ExpInt-def} that $F(x)\ge0$ for $x>0$, the obvious inequality $1/z_{T}\le\mu^2 A_{T}/2$, and some elementary algebra, it follows that
\begin{equation*}
(0<)\;
J_1(T)
\triangleq
F\left(\dfrac{2}{\mu^2 A_{T}}\right)
-
F\left(\dfrac{2}{\mu^2 T}\right)
\le
\dfrac{A_T}{T}-1
\le
\dfrac{1}{\abs{\mu}\sqrt{T}}\to 0,
\end{equation*}
as $T\to+\infty$.

To see that $J_2(T)\to0$ as $T\to+\infty$, it suffices to appeal to~\eqref{eq:SRP-AT-dbl-ineq} and to~\eqref{eq:LwrBnd-last-int-term-asymp}, which combined yield the desired conclusion right away, because, by definition, $J_2(T)>0$ for all $T>0$. To be more specific, by the First Mean Value Theorem for definite integrals we obtain:
\begin{equation*}
\bigintsss_{T}^{A_{T}}F\left(\dfrac{2}{\mu^2 x}\right)\dfrac{dx}{x}
=
F(z_{T})\log\left(\dfrac{A_{T}}{T}\right)
\;\;
\text{for some}
\;\;
z_{T}\in\left(\dfrac{2}{\mu^2 A_{T}},\dfrac{2}{\mu^2 T}\right),
\end{equation*}
whence, in view of~\eqref{eq:SRP-AT-dbl-ineq}, and because again $(0<)\;e^x\E1(x)\le 1/x$ for $x>0$, as given by~\cite[Inequality~5.1.19,~p.~229]{Abramowitz+Stegun:Handbook1964}, and $1/z_{T}\le\mu^2 A_{T}/2$, it follows that
\begin{multline*}
(0<)\;J_2(T)
\le
\dfrac{A_{T}}{T}\log\left(\dfrac{A_{T}}{T}\right)\\
\le
\left(1+\dfrac{1}{\abs{\mu}\sqrt{T}}\right)\log\left(1+\dfrac{1}{\abs{\mu}\sqrt{T}}\right)
\le
\left(1+\dfrac{1}{\abs{\mu}\sqrt{T}}\right)\dfrac{1}{\abs{\mu}\sqrt{T}}\to0,
\end{multline*}
as $T\to+\infty$.

Now that it is clear that $J_1(T)\to0$ and $J_2(T)\to0$ as $T\to+\infty$, establishing~\eqref{eq:SRP-opt-order3}, which is the desired order-three Pollak-minimaxity of the SRP procedure, is a merely matter of putting all of the above together. As an aside we note that, from our above analysis, it is clear that $J_1(T)$ and $J_2(T)$ both go to $0$ as $T\to+\infty$ no slower than $1/\sqrt{\mu^2 T}$. Hence the SRP procedure's delay risk $\overline{C}(\bar{\T}_{A_{T}}^{*})$ decays down to the lowerbound $\overline{B}(T)$ no slower than $1/\sqrt{\mu^2 T}$. This is a conservative estimate, and its improvement would require obtaining a tighter version of the double inequality~\eqref{eq:lambda-dbl-ineq}, and subsequently also refining the assertion of Lemma~\ref{lem:SRP-AT-dbl-ineq}. It would also require obtaining a tighter upperbound on the quasi-stationary cdf $Q_A(x)$ given by~\eqref{eq:QST-cdf-answer}. Recall that in the above analysis we used the trivial upperbound $Q_A(x)\le1$ which, by definition, is true for any cdf. To get a tighter upperbound, the high-order large-$A$ approximations obtained in~\cite{Polunchenko:SA2017} for the quasi-stationary distribution might come in handy. However, this is beyond the scope of this paper, and the corresponding analysis will be carried out elsewhere.

We conclude with an illustration of the obtained result, viz.~\eqref{eq:SRP-opt-order3}, at work. Specifically, we wrote a Mathematica script that numerically evaluates the delay risk $\overline{C}(\bar{\T}_{A_{T}}^{*})$ and the lowerbound $\overline{B}(T)$ via formulae~\eqref{eq:LwrBnd-SR-formula} and~\eqref{eq:SRP-SADD-LwrBnd-form}, respectively. The script allows to compute $\overline{C}(\bar{\T}_{A_{T}}^{*})$ and $\overline{B}(T)$ to within any desired accuracy, although each additional decimal place of accuracy clearly comes at the ``price'' of slower speed of computation. As a reasonable compromise, we went with ten decimal places, which is more than sufficient for our purposes, and yet, on an average laptop, the amount of time it takes the script to finish is on the order of seconds. The value of $\abs{\mu}>0$ is a factor as well: the computational time is lesser the higher the value of $\abs{\mu}$. This makes sense because the pre- and post-change hypotheses are harder to differentiate between when $\abs{\mu}$ is small. We experimented with two scenarios: $\mu=1/2$, which is a relatively small (harder to detect) change, and $\mu=1$, which is a more contrast (easier to detect) change. For each of the two values of $\mu$ the experiment consisted in varying the ARL to false alarm level $T$ from $1$ up through $100$ in increments of $1$, and using the script to compute $\overline{C}(\bar{\T}_{A_{T}}^{*})$ and $\overline{B}(T)$ for each $T$. The threshold $A_{T}>0$ required for the evaluation of $\overline{C}(\bar{\T}_{A_{T}}^{*})$ was recovered numerically from equation~\eqref{eq:Whit1-eigval-eqn} using the high-order approximations obtained in~\cite{Polunchenko:SA2017}. All of the obtained experimental results are shown in Figures~\ref{fig:C_and_B_vs_T__mu_1_over_2} and~\ref{fig:C_and_B_vs_T__mu_1}. Specifically, either figure is a pair of graphs arranged side by side: one showing $\overline{C}(\bar{\T}_{A_{T}}^{*})$ and $\overline{B}(T)$ together in one plot, and one showing the corresponding difference $\overline{C}(\bar{\T}_{A_{T}}^{*})-\overline{B}(T)$ in a separate plot---all as functions of $T\in[1,100]$. Figure~\ref{fig:C_and_B_vs_T__mu_1_over_2} corresponds to $\mu=1/2$, and Figure~\ref{fig:C_and_B_vs_T__mu_1} corresponds to $\mu=1$.
\begin{figure}[tbhp]
    \centering
    \subfloat[$\overline{C}(\bar{\T}_{A_{T}}^{*})$ and $\overline{B}(T)$ vs. $T$]{
        \includegraphics[width=0.47\textwidth]{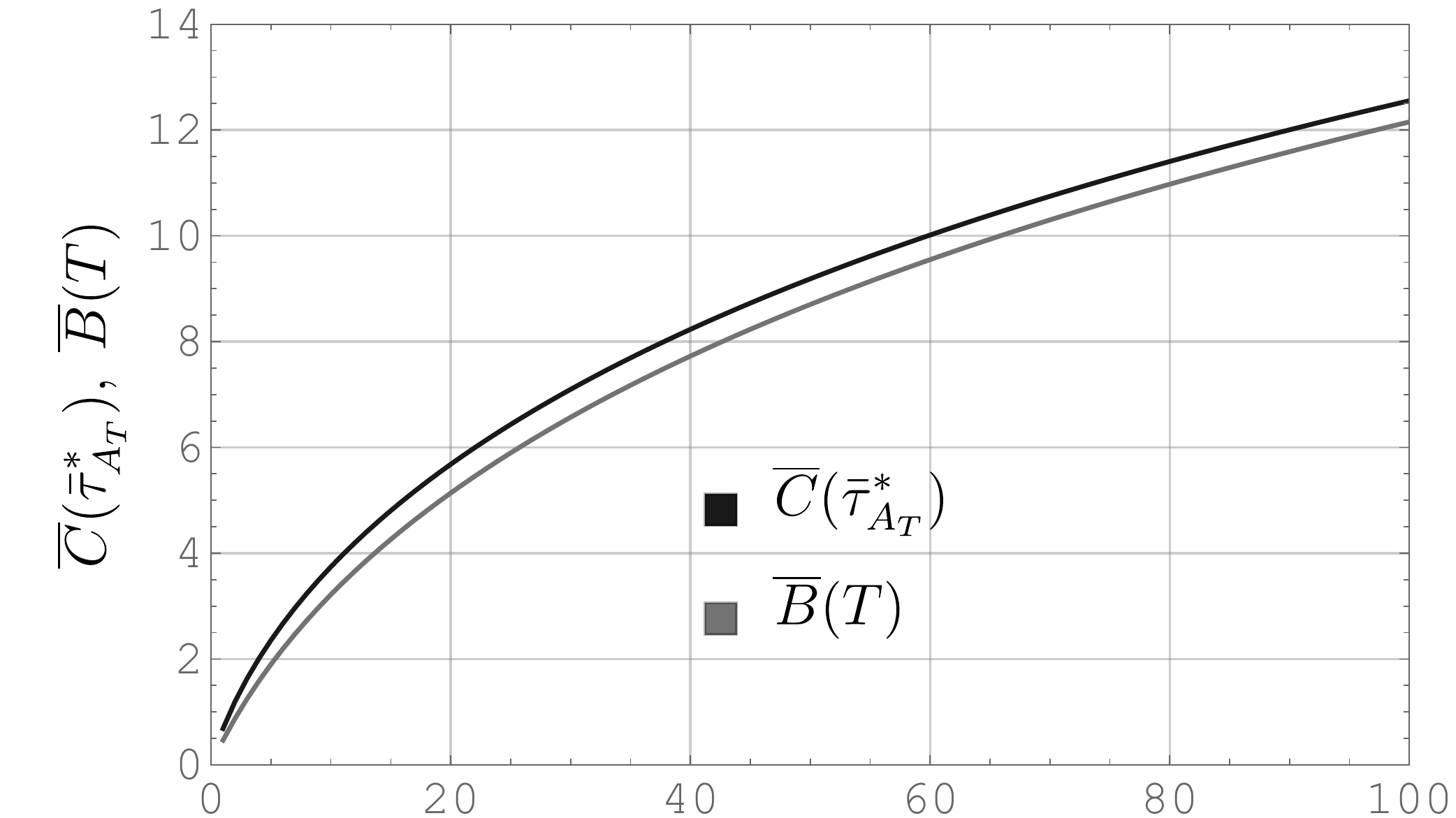}
        }
    \;
    \subfloat[$\overline{C}(\bar{\T}_{A_{T}}^{*})-\overline{B}(T)$ vs. $T$]{
        \includegraphics[width=0.47\textwidth]{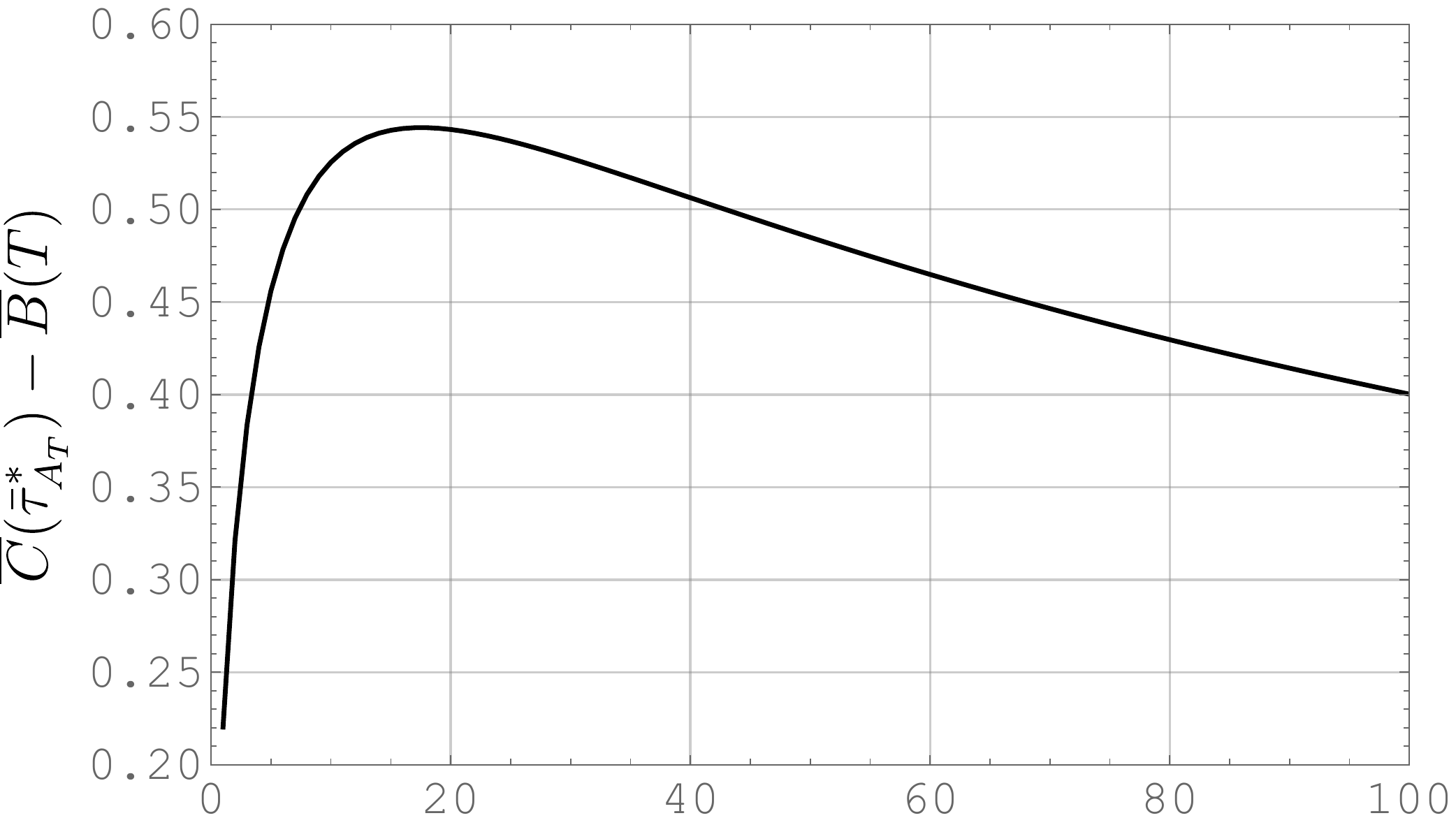}
        }
    \caption{Performance of the SRP procedure $\overline{C}(\bar{\T}_{A_{T}}^{*})$ and the lowerbound $\overline{B}(T)$ as functions of $T\in[0,100]$ for $\mu=1/2$.}
    \label{fig:C_and_B_vs_T__mu_1_over_2}
\end{figure}

A visual inspection of the figures suggests two conclusions to draw. First, it is fairly evident that $\overline{C}(\bar{\T}_{A_{T}}^{*})$ does, in fact, converge to $\overline{B}(T)$ from above. This is exactly what one would expect in view of~\eqref{eq:SRP-opt-order3}. Second, the convergence is slower for $\mu=1/2$ than for $\mu=1$, which is also an expected result, because fainter changes are more difficult to detect, so that $\overline{C}(\bar{\T}_{A_{T}}^{*})$ and $\overline{B}(T)$ are both larger, and the difference between the two is more pronounced as well. We also experimented with ramping up the ARL to false alarm level $T$ to as high as $10,000$ and setting $\mu$ as low as $1/10$, and obtained sufficiently convincing numerical evidence that $\overline{C}(\bar{\T}_{A_{T}}^{*})$ does eventually ``blend in'' with $\overline{B}(T)$, even if $\mu$ is small. 
\begin{figure}[tbhp]
    \centering
    \subfloat[$\overline{C}(\bar{\T}_{A_{T}}^{*})$ and $\overline{B}(T)$ vs. $T$]{
        \includegraphics[width=0.47\textwidth]{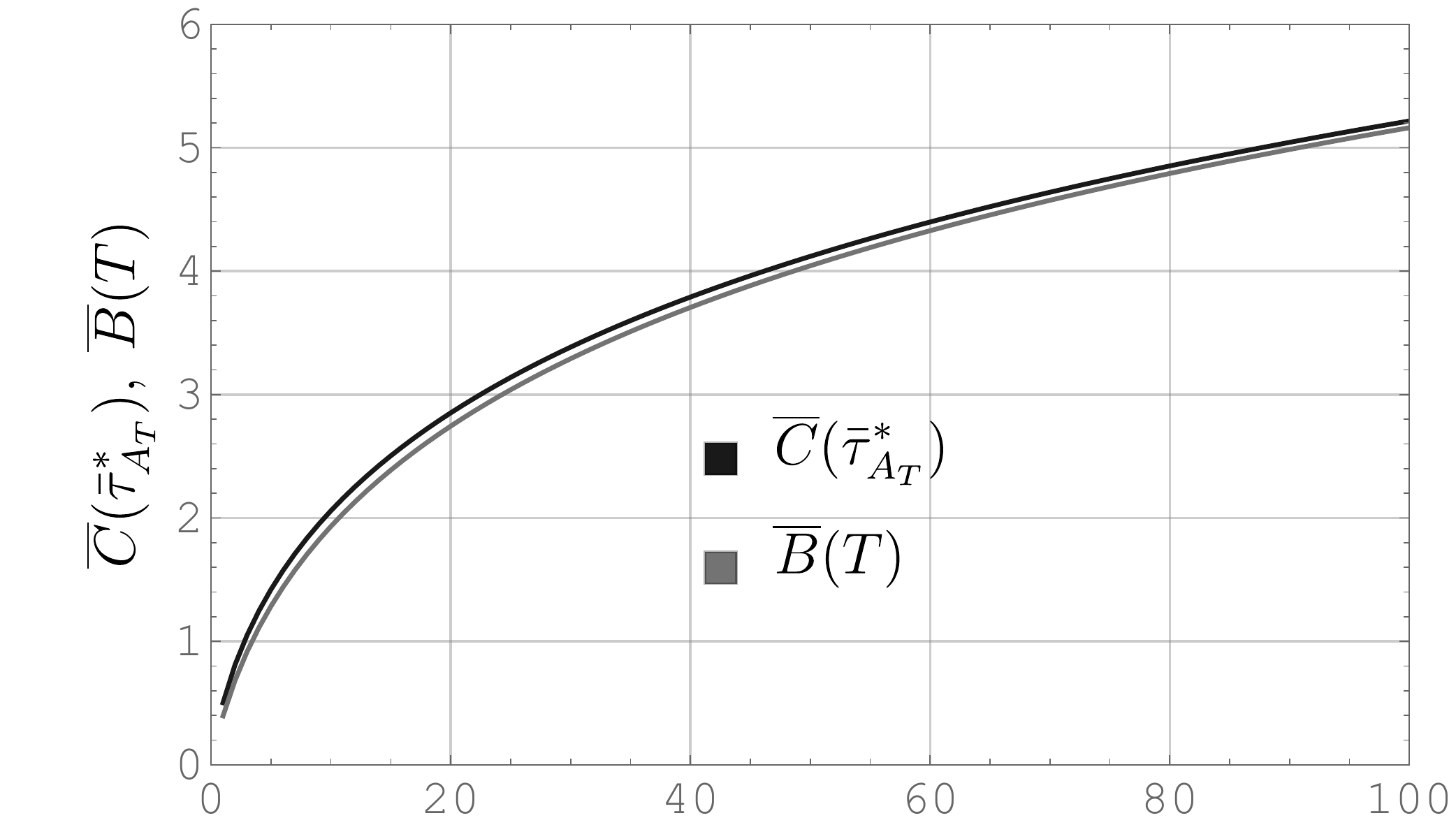}
        }
    \;
    \subfloat[$\overline{C}(\bar{\T}_{A_{T}}^{*})-\overline{B}(T)$ vs. $T$]{
        \includegraphics[width=0.47\textwidth]{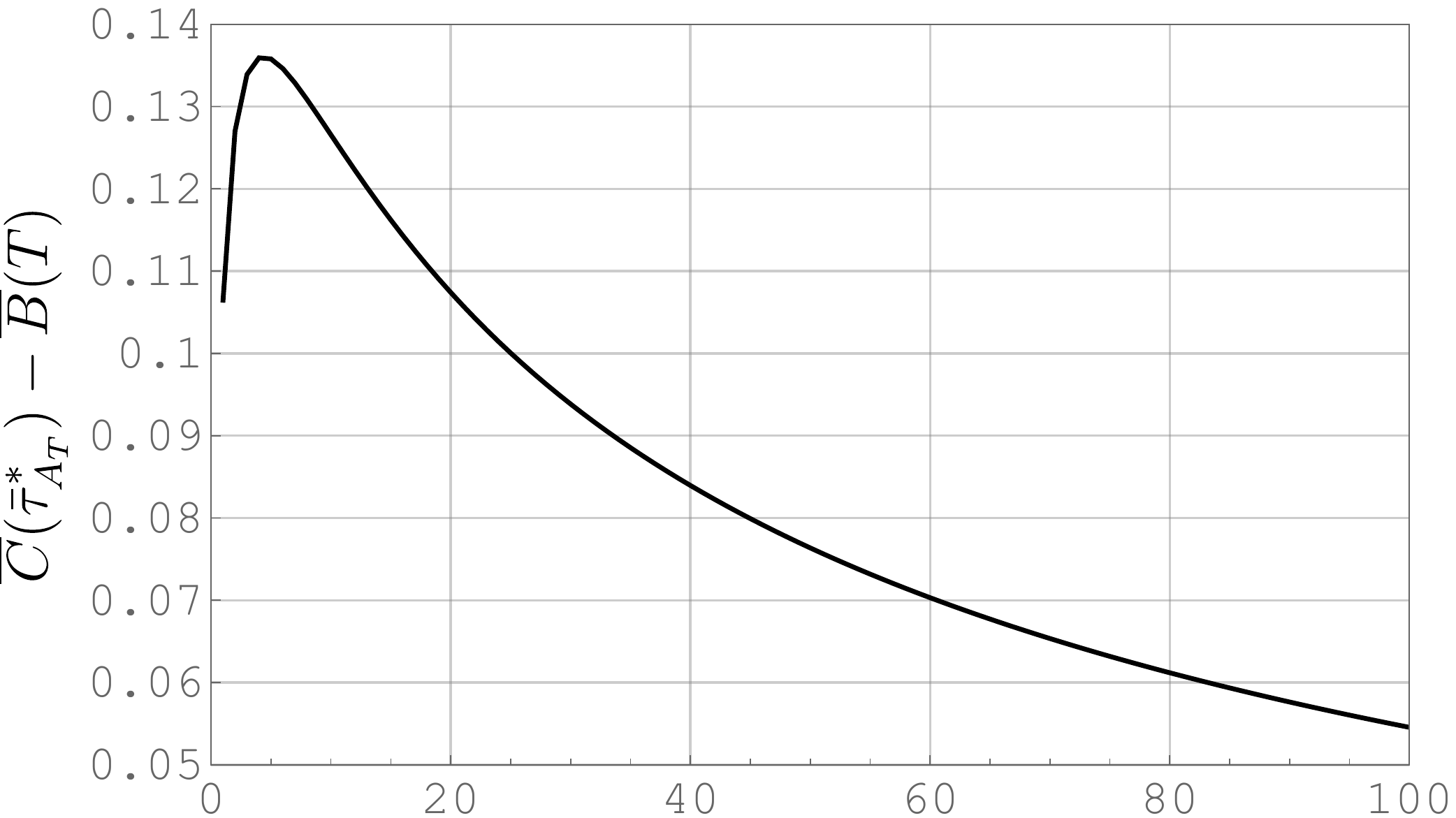}
        }
    \caption{Performance of the SRP procedure $\overline{C}(\bar{\T}_{A_{T}}^{*})$ and the lowerbound $\overline{B}(T)$ as functions of $T\in[0,100]$ for $\mu=1$.}
    \label{fig:C_and_B_vs_T__mu_1}
\end{figure}

\section*{Acknowledgments}
The author is thankful to Dr.~E.V.~Burnaev of the Kharkevich Institute for Information Transmission Problems, Russian Academy of Sciences, Moscow, Russia, and to Prof.~A.N.~Shiryaev of the Steklov Mathematical Institute, Russian Academy of Sciences, Moscow, Russia, for the interest and attention to this work.

\bibliographystyle{siamplain}
\bibliography{main,special-functions,stochastic-processes}

\end{document}